\documentclass[12pt]{amsart}

\usepackage{amssymb, tikz}

\usepackage{hyperref}

\usepackage{amsmath}

\usepackage{amsmath, amsthm, amscd, amsfonts}

\usepackage{mathtools}

\usepackage[english]{babel}
\usepackage{setspace}
\usepackage[utf8]{inputenc}

\usepackage{mathrsfs}

\newtheorem{theorem}{Theorem}[section]
\newtheorem{claim}[theorem]{Claim}

\newtheorem{lemma}[theorem]{Lemma}

\theoremstyle{definition}
\newtheorem{definition}[theorem]{Definition}

\theoremstyle{remark}

\newtheorem{notation}[theorem]{Notation}

\newtheorem{hypotheses}[theorem]{Hypotheses}

\newcommand{\dom}{{\rm dom}}

\newcommand{\rng}{{\rm rng}}
\newcommand{\Lim}{{\rm Lim}}

\newcommand{\cH}{{\mathscr H}}

\newcommand{\cI}{{\mathscr I}}

\newcommand{\MPB}{{\mathbb P}}

\newcommand{\Gen}{{\rm Gen}}

\newcommand{\CH}{{\rm CH}}

\def\MPB{{\mathbb{P}}}
\def\MQB{{\mathbb{Q}}}
\def\MRB{{\mathbb{R}}}

\newcount\skewfactor
\def\mathunderaccent#1#2 {\let\theaccent#1\skewfactor#2
\mathpalette\putaccentunder}
\def\putaccentunder#1#2{\oalign{$#1#2$\crcr\hidewidth
\vbox to.2ex{\hbox{$#1\skew\skewfactor\theaccent{}$}\vss}\hidewidth}}

\usepackage{hyperref}

\begin{document}

\title {Adding highly undefinable sets over $L$}

\author[M.  Golshani]{Mohammad Golshani}

\address{Mohammad Golshani, School of Mathematics, Institute for Research in Fundamental Sciences (IPM), P.O.\ Box:
	19395--5746, Tehran, Iran.}

\email{golshani.m@gmail.com}
\urladdr{http://math.ipm.ac.ir/~golshani/}

\author[S. Shelah]{Saharon Shelah}
\address{Einstein Institute of Mathematics,
The Hebrew University of Jerusalem,
9190401, Jerusalem, Israel; and\\
Department of Mathematics,
Rutgers University,
Piscataway, NJ 08854-8019, USA}
\urladdr{https://shelah.logic.at/}
\email{shelah@math.huji.ac.il}
\thanks{ The first author's research has been supported by a grant from IPM (No. 1401030417). The second author thanks an individual who wishes to remain anonymous for generously funding typing services.  The
	second author would like to thank the Israel Science Foundation (ISF) for partially supporting this research by grant No 1838/19.
This is publication number
1227
in Saharon Shelah's list.}

\subjclass[2020]{Primary: 03E35, 03E45, 03E55}

\keywords {proper forcing,  undefinability.}


\begin{abstract}
We sow that there exists a generic extension of the G\"{o}del's constructible universe in which diamond holds and there exists a subset  $Y \subseteq \omega_1$ such that for stationary many $\delta < \omega_1,$ the set
 $Y \cap \delta$ is not definable in the structure $(L_{F(\delta)}, \in)$, where $F(\delta) > \delta$ is the least ordinal
 such that $L_{F(\delta)}\models$``$\delta$ is countable''.
\end{abstract}

\maketitle
\numberwithin{equation}{section}
\section{introduction}

In this short note we introduce a $\Sigma^2_2$ sentence $\phi$ which is false in $L$, the G\"{o}del's constructible universe. Furthermore, we force
 a proper generic extension of $L$ in which  both $\phi$ and diamond hold.
 More precisely, let $\phi$ be the following sentence.
 \begin{definition}
Let $\phi$ be the sentence
\[
(\exists R,   S_1, S_2, F, Y) \bigwedge\limits_{i=1}^3 \varphi_i,
\]
where
\begin{enumerate}
\item $\varphi_1:=\varphi_1(R,  F)$ is the conjunction of the following statements:
\begin{enumerate}

\item $(\omega_1, R)$ is isomorphic to $(L_{\omega_1}, \in)$,

\item Let $\Lim(\omega_1)$ be the set of countable limit ordinals. Then $(\Lim(\omega_1), R)$ is isomorphic to
$(\Lim(\omega_1), \in)$,

\item $F$ is a function defined on $\Lim(\omega_1)$, such that for every limit ordinal $\delta$,
\begin{center}
$(\omega_1, R) \models$$\ulcorner F(\delta)$ is the minimal limit ordinal such that $L_{F(\delta)} \models$``$|\delta|=\aleph_0$''$\urcorner$.
\end{center}

\end{enumerate}
\item $\varphi_2:=\varphi_2(S_1, S_2)$ is the statement: $S_1$ and $S_2$ form a partition of $\Lim(\omega_1)$ into disjoint stationary sets.

\item $\varphi_3:=\varphi_3(R,   S_2, F, Y)$ is the statement:
\begin{center}
$(\forall  \delta \in S_2) \big( Y \cap \delta$ is not definable in $(L_{F(\delta)})^{(\omega_1, R)} \big)$.
\end{center}
\end{enumerate}
\end{definition}
 It is clear that $\phi$ is a $\Sigma^2_2$ statement. We prove the following theorem.
 \begin{theorem}
 \label{th1}
 \begin{enumerate}
 \item[(a)] The statement $\phi$ fails in $L$,

 \item[(b)] There exists a proper forcing notion $\MPB \in L$ which forces $\Diamond+\phi.$
 \end{enumerate}
 \end{theorem}

 We assume familiarity with proper forcing notions. Given a forcing notion
$\MPB$ and conditions $p, q \in \MPB$, by $p \leq q$ we mean $q$ is stronger than $p.$

\section{Some preliminaries}
\label{prel}
In this paper, we are interested in forcing notions which preserve diamond at $\omega_1$.
Let us recall the definition of a diamond sequence.
\begin{definition}
\label{diamond} (\cite{jensen}) Assume $S \subseteq \omega_1$
is stationary. Then $\Diamond_S$ asserts the existence of a sequence $\langle s_\alpha: \alpha \in S       \rangle$
such that $s_\alpha \subseteq \alpha$, for $\alpha \in S$, and for every $X \subseteq \omega_1$,
the set
$\{ \alpha \in S: X \cap \alpha=s_\alpha\}$ is stationary in $\omega_1$. By $\Diamond$ we mean $\Diamond_{\omega_1}$.
\end{definition}
By the work of Jensen \cite{jensen}, $\Diamond_S$ holds in the G\"{o}del's constructible universe, for all stationary subsets $S$ of $\omega_1$.
We now introduce a property of forcing notions which is sufficient to guarantee that $\Diamond$ is preserved, see Lemma \ref{f21}.
\begin{definition}
(\cite[Ch. V, Definition 1.1]{Sh:f})
\label{f19}
Suppose $S \subseteq \omega_1$ is stationary,  $\MPB$ is a forcing notion and $N$ is a countable model with $\MPB \in N.$
\begin{enumerate}
\item The sequence $\langle p_n: n<\omega \rangle$ is a generic sequence for $(N, \MPB)$ if it is an increasing sequence from $ \MPB \cap N$ and for every dense open subset $D$ of $\MPB$ in $N$, $D \cap \{p_n: n<\omega    \} \neq \emptyset.$

\item The pair $(N, \MPB)$ is complete if every generic sequence $\langle p_n: n<\omega \rangle$ for $(N, \MPB)$ has an upper bound in $\MPB$.

\item We say $\MPB$ is $\{ S\}$-complete if for every large enough regular $\chi$ and  countable model
$N \prec (\cH(\chi), \in)$, if $S, \MPB \in N$ and $N \cap \omega_1 \in S$, then the pair $(N, \MPB)$ is complete.
\end{enumerate}
\end{definition}

\begin{lemma}
\label{f21} (\cite[ Chapter V, Claim 1.9]{Sh:f})
Suppose $S\subseteq \omega_1$ is stationary.
Assume $\MPB$ is $\{S\}$-complete. If $\Diamond_S$
holds in $V$, then it  holds in $V^{\MPB}$ as well.
\end{lemma}

\section{Proof of the main theorem}
\label{mainproof}
In this section we prove Theorem \ref{th1}.
\subsection{$\phi$ fails in $L$.}
In this subsection we show that the statement $\phi$ is false in the constructible universe $L$. This follows from the following lemma.
\begin{lemma}
\label{lem1} Assume $V=L.$ Let $S\subseteq \omega_1$ be stationary. Then for every
$X \subseteq \omega_1$, there exists $\delta \in S$ such that $X \cap \delta$ is definable in $(L_{F(\delta)}, \in)$.
\end{lemma}
\begin{proof}
Let us recall the construction of a $\Diamond_S$-sequence. By induction on $\delta$ we define a sequence
$\langle  (s_\delta, c_\delta): \delta \in S        \rangle$
 as follows.
Suppose $\delta \in S$ and we have defined $(s_\gamma, c_\gamma)$,
 for $\gamma \in S \cap \delta$ such that for each $\gamma,$ $s_\gamma \subseteq \gamma$
 and $c_\gamma \subseteq \gamma$
 is a club. If there exists a pair $(s, c)$ such that:
 \begin{enumerate}
 \item $s \subseteq \delta,$

 \item $c \subseteq \delta$ is a club,

 \item for all $\gamma \in  c \cap S, s \cap \gamma \neq s_\gamma$,
 \end{enumerate}
then let $(s_\delta, c_\delta)$ be the $<_L$-least such pair. Otherwise set $s_\delta=\emptyset$
and $c_\delta=\delta$.
\begin{claim}
\label{cl1}
There exists a club $C$ of $\omega_1$ such that
\[
C \cap S \subseteq \{\delta \in S: (s_\delta, c_\delta) \text{~is  definable in ~} (L_{F(\delta)}, \in)  \}.
\]
\end{claim}
\begin{proof}
To see this, suppose by the way of contradiction there is no such club $C$. It then follows that the set
$$S_*=\{\delta \in S: (s_\delta, c_\delta) \text{~is not  definable in ~} (L_{F(\delta)}, \in)  \}$$
is stationary. Let $M$ be a countable elementary submodel of $(L_{\omega_2}, \in)$, such that $M$ contains all relevant information
and $M \cap \omega_1=\mu \in S_*.$
Let also $\pi: M \simeq L_\delta$ be the transitive collapse map. Then
\begin{itemize}
\item $\pi(\omega_1)=\mu$

\item $\pi(S)=S \cap \mu,$ and $\pi(S_*)=S_* \cap \mu,$

\item $\pi(\langle  (s_\delta, c_\delta): \delta \in S        \rangle)=\langle  (s_\gamma, c_\gamma): \gamma \in S \cap \mu        \rangle$.
\end{itemize}
As $\mu \in S_*,$ it follows from the definition of $S_*$ that
$ (s_\mu, c_\mu)$ is not  definable in  $(L_{F(\mu)}, \in).$

On the other hand,
$(s_\mu, c_\mu)$ is uniformly definable using the sequence
$\langle  (s_\gamma, c_\gamma): \gamma \in S \cap \mu        \rangle$,
hence $(s_\mu, c_\mu)$ is definable in  $L_{\delta}.$
Now note that $(L_{\delta}, \in) \models$``$\mu$ is uncountable'',
hence $F(\mu) > \delta$, and thus $(s_\mu, c_\mu)$ is definable in  $(L_{F(\mu)}, \in)$, a contradiction.
\end{proof}

Now suppose that $X \subseteq \omega_1$ is given.
It follows that the set
\[
T=\{\delta \in S: X \cap \delta=s_\delta  \}
\]
is stationary in $\omega_1$. Let $\delta \in C \cap T.$ It then follows that  $X \cap \delta=s_\delta$
and $s_\delta \text{~is  definable in ~} (L_{F(\delta)}, \in)$. Thus
$X \cap \delta \text{~is  definable in ~} (L_{F(\delta)}, \in)$, as requested.
\end{proof}

\subsection{Consistency of $\Diamond+\phi$}
We now show that in some forcing extension of $L$, $\Diamond+\phi$ holds.
\begin{lemma}
\label{lem3} Assume $V=L,$ and let  $S_1, S_2$ be  a partition of $\Lim(\omega_1)$
  into disjoint stationary sets. Suppose $\varphi_1(R,   F)$ holds. Then there exists a $\{ S_1\} $-complete proper forcing notion $\mathbb{P}$,
   which adds a set $Y$ witnessing $\varphi_3(R,  S_2, F, Y)$ holds.
  In particular,
     $\Diamond+\phi$
holds in $L[G_{\mathbb{P}}]$.
\end{lemma}
\begin{proof}
Let $\mathbb{P}$ be the set of all conditions  $p$ where:
\begin{enumerate}
\item[$(\ast)_1$] $p$ is a countable subset of $\omega_1$,

\item[$(\ast)_2$] $\max(p)$ exists,

\item[$(\ast)_3$] for all $\delta \in S_2 \cap (\max(p)+1), p \cap \delta$ is not definable in $ (L_{F(\delta)})^{(\omega_1, R)}.$
\end{enumerate}
Given two conditions $p, q$ let us say that $p \leq q$ ($q$ is stronger than $p$), iff $p =q \cap (\max(p)+1)$.

\begin{claim}
\label{claim2}
$\mathbb{P}$ is proper.
\end{claim}
\begin{proof}
Suppose $\chi$ is large enough regular and
$N \prec (\cH(\chi), \in)$ is countable such that:
\begin{itemize}
 \item $N= \bigcup_{n<\omega} N_n$, where $\langle N_n: n<\omega \rangle$ is a $\prec$-increasing sequence
 of elementary submodels of $N$ with $\{N_n: n<\omega \} \subseteq N,$
\item $ R, F, S_1, S_2, \MPB \in N_0$.\footnote{Note that the class of all countable models $N \in [\cH(\chi)]^{\aleph_0}$  as above forms a club of
$[\cH(\chi)]^{\aleph_0}$, thus it suffices to check properness with respect to such models.}
\end{itemize}
 Let also $p \in \MPB \cap N$.  We have to find $q \geq p$
 which is an $(N, \MPB)$-generic condition. We may assume that $p \in N_0.$

 Let $\delta=N \cap \omega_1$ and for each $n<\omega$ set $\delta_n=N_n \cap \omega_1$.
Choose an increasing $\omega$-sequence $\eta_\delta=\langle \eta_\delta(n): n<\omega     \rangle$, definable  in $(L_{F(\delta)})^{(\omega_1, R)}$, coding
 a cofinal sequence in $\delta$ with $\eta_\delta(0) > \max(p)$. Such a sequence exists by the choice of $F(\delta).$
 Let also $c_\delta=\langle c_\delta(n): n<\omega    \rangle$
 be a real, not definable in  $(L_{F(\delta)})^{(\omega_1, R)}$.

Now let $\langle \mathcal{D}_n: n<\omega        \rangle$
be an enumeration of dense open subsets of $\mathbb{P}$ in $N$. Following ideas from \cite{shelah77}, we  define an increasing sequence
 $\langle p_m: m<\omega   \rangle$ of conditions such that:
 \begin{enumerate}
 \item $p_0=p,$

 \item For all $n< \omega$, there exists $m<\omega$ such that $p_{m} \in \mathcal{D}_n,$

 \item For all  $n<\omega,$
 \[
 \eta_{\delta}(n) \in \bigcup_{m<\omega}p_m~~ \iff~~ c_{\delta}(n)=1.
 \]
 \end{enumerate}
 To start set $p_0=p$. Note that $p \cap \{ \eta_\delta(n): n<\omega \}=\emptyset.$ Let us define $p_1$.
 Let $k_1< \omega$ be such that $\{\eta_\delta(n): n<\omega   \} \cap N_{1}= \{\eta_\delta(n): n<k_1   \}$, and let $\mu_1 > \delta_0$ be such that $\sup\{\eta_\delta(n): n<k_1   \} < \mu_1 < \delta_1$ with $\mu_1 \notin S_2$. Set
\begin{center}
 $q_1 = p_0 \cup \{\eta_\delta(n): n< k_1$ and $c_\delta(n)=1     \} \cup \{\mu_1\}$.
 \end{center}
 Note that $q_1 \in \MPB \cap N_1$. Now let $p_1$ be such that:
 \begin{enumerate}
 \item[(4)] $p_1 \in N_1,$
 \item[(5)] $p_1 \geq q_1$,
 \item[(6)] $p_1 \in \bigcap\{ \mathcal{D}_n: n< 1$ and $ \mathcal{D}_n \in N_1\}.$ If  $\mathcal{D}_0 \notin N_1$, set $p_1=q_1.$
 \end{enumerate}
 Now suppose that $1 \leq m<\omega$ and we have defined $p_m$. We define $p_{m+1}.$
 Let $k_{m+1} \geq k_m$ be such that
 $\{\eta_\delta(n): n<\omega   \} \cap (N_{m+1} \setminus N_m)= \{\eta_\delta(n): k_m \leq n<k_{m+1}   \}$, and let $\mu_{m+1} > \delta_m$ be such that $\sup\{\eta_\delta(n): n<k_{m+1}   \} < \mu_{m+1} < \delta_{m+1}$ with $\mu_{m+1} \notin S_2$.
 Set
\begin{center}
 $q_{m+1} = p_m \cup \{\eta_\delta(n): k_m \leq n< k_{m+1}$ and $c_\delta(n)=1     \} \cup \{\mu_{m+1}\}$.
 \end{center}
 Note that $q_{m+1} \in  \MPB \cap N_{m+1}$. Now let $p_{m+1}$ be such that:
 \begin{enumerate}
 \item[(7)] $p_{m+1} \in N_{m+1},$
 \item[(8)] $p_{m+1} \geq q_{m+1}$,
 \item[(9)] $p_{m+1} \in \bigcap\{ \mathcal{D}_n: n< m+1$ and $ \mathcal{D}_n \in N_{m+1}\}.$ If there are no such  $\mathcal{D}_n$'s, set $p_{m+1}=q_{m+1}.$
 \end{enumerate}
 Set $p=\bigcup_{m<\omega}p_m \cup \{\delta\}$. We claim that $p \in \MPB$. To show $p$ is a condition, it suffices to show that $p \cap \delta= \bigcup_{m<\omega}p_m$
is not definable in $(L_{F(\delta)})^{(\omega_1, R)}$. Suppose by the way of contradiction that $p \cap \delta$ is definable in $ (L_{F(\delta)})^{(\omega_1, R)}$.
As the sequence $\eta_\delta$ is definable in $ (L_{F(\delta)})^{(\omega_1, R)}$, it follows from clause (3) that
$\langle  c_\delta(n): n <\omega   \rangle$ is definable in $ (L_{F(\delta)})^{(\omega_1, R)}$, which is a contradiction.
 It is clear from our construction that
$p$  is an $(N, \MPB)$-generic condition.
\end{proof}
The above proof implies the following.
\begin{claim}
\label{claim1}
Assume $p \in \mathbb{P}$ and $\gamma > \max(p)$. Then there exists a condition  $q \geq p$
such that $\max(q) \geq \gamma.$
\end{claim}
The next claim guarantees that $\Diamond_{S_1}$ is preserved by $\MPB.$
\begin{claim}
\label{claim3}
$\mathbb{P}$ is $\{ S_1\} $-complete.
\end{claim}
\begin{proof}
Suppose $\chi$ is large enough regular and
$N \prec (\cH(\chi), \in)$ is countable such that $R, S_1, S_2, F, \MPB \in N$ and $\delta=N \cap \omega_1 \in S_1$. We show that the pair $(N, \MPB)$ is complete.
Thus let $\langle p_n: n<\omega \rangle$ be a generic sequence for $(N, \MPB)$. Set  $p=\bigcup_{n<\omega}p_n \cup \{\delta\}$.  Note that, by Claim
\ref{claim1}, $\sup\bigcup_{n<\omega}p_n=\delta$,
and since $\delta \notin S_2$, $p \in \mathbb{P}$. Then $p$ is an upper bound
in $\mathbb{P}$ for the sequence $\langle p_n: n<\omega \rangle$.
\end{proof}
Now let $G$ be $\mathbb{P}$-generic over $V$ and let $Y=  \bigcup \{p: p \in G  \}.$
\begin{claim}
\label{claim4}
The set $Y$ witnesses that $\varphi_3(R,   S_2, F, Y)$ holds in $V[G]$.
\end{claim}
\begin{proof}
This is clear.
\end{proof}
 It follows from Lemma \ref{f21} and Claim
\ref{claim3} that $\Diamond_{S_1}$ holds in $L[G]$. Finally note that
$\phi$ holds in $V[G]$:
\begin{itemize}
\item $\varphi_1(R,   F)$ holds by the choice of $R$ and $F$.

\item By Claim \ref{claim2},
$S_1$ and $S_2$ remain stationary in $L[G].$ It then follows that $\varphi_2(S_1, S_2)$ holds in $L[G]$ as well.


\item $\varphi_3(R,   S_2, F, Y)$ holds,   by Claim \ref{claim4}.
\end{itemize}
The lemma follows.
\end{proof}
\begin{proof}[Proof of Theorem \ref{th1}]
It follows from Lemmas \ref{lem1} and \ref{lem3}.
\end{proof}

\end{document}